\documentclass[12pt,final]{amsart}
\usepackage{amssymb,amsfonts,amsmath,amsthm}
\usepackage{latexsym}
\usepackage{array}
\usepackage{amsbsy}
\usepackage{stmaryrd}

\usepackage[asymmetric, hmarginratio=2:3]{geometry}
\usepackage[colorlinks=true, citecolor=blue, linkcolor=blue, urlcolor=blue]{hyperref}
\usepackage{upref}
\usepackage{amssymb}

\def\p{\partial}

\def\R{\mathbb R}

\def\P{\mathbb P}
\def \mcT{\mathcal{T}_h}
\usepackage{mathtools}

\DeclarePairedDelimiter\norm{\lVert}{\rVert}
\DeclarePairedDelimiter{\seminorm}{[}{]}
\DeclarePairedDelimiter{\jump}{\llbracket}{\rrbracket}
\DeclarePairedDelimiter{\tnorm}{\interleave}{\interleave}

\mathtoolsset{showonlyrefs}

\usepackage{ifdraft}
\ifoptionfinal{
\usepackage[disable]{todonotes}
}{
\date{Compiled \today}
\usepackage[breakall, fit]{truncate}
\usepackage{accsupp}
\setlength{\marginparsep}{5pt}

\usepackage[notref, notcite]{showkeys}
\usepackage[bordercolor=white, color=white, textwidth=100pt]{todonotes}
}

\newcommand{\VN}{\mathcal{V}_N}

\makeatletter\providecommand\@dotsep{5}\def\listtodoname{List of Todos}\def\listoftodos{\hypersetup{linkcolor=black}\@starttoc{tdo}\listtodoname\hypersetup{linkcolor=blue}}\makeatother

\newtheorem{theorem}{Theorem}[section]
\newtheorem{lemma}[theorem]{Lemma}
\newtheorem{proposition}[theorem]{Proposition}

\theoremstyle{definition}

\theoremstyle{remark}
\newtheorem{remark}[theorem]{Remark}

\title[FEM for UC with finite dimensional trace]{finite element approximation of unique continuation of functions with finite dimensional trace}
\author[E. Burman]{Erik Burman}
\address{Department of Mathematics, 
University College London, 
Gower Street, London UK, WC1E 6BT.}
\email{e.burman@ucl.ac.uk}
\author[L. Oksanen]{Lauri Oksanen}
\address{Department of Mathematics and Statistics, University of Helsinki, P.O 68, 00014 University of Helsinki, Finland}
\email{l.oksanen@ucl.ac.uk}


\begin{document}
\begin{abstract}
We consider a unique continuation problem where the Dirichlet trace of the solution is known to have finite dimension. We prove Lipschitz stability of the unique continuation problem and design a finite element method that exploits the finite dimensionality to enhance stability. Optimal a priori and a posteriori error estimates are shown for the method. The extension to problems where the trace is not in a finite dimensional space, but can be approximated to high accuracy using finite dimensional functions is discussed. Finally, the theory is illustrated in some numerical examples.
\end{abstract}
\maketitle
\ifoptionfinal{}{
}
\section{Introduction}

Unique continuation (UC), broadly speaking is the unique extension of a function from some domain $\omega$ to a larger domain $\Omega$ subject to it being the solution of a given elliptic partial differential equation. It has applications in data assimilation, inverse problems and control. In spite of the many applications computational UC has received comparatively little attention from the numerical analysis community. 

After a suitable regularization on the continuous level, the UC problem is well-posed and can, in principle, be handled using standard methods, however the approximation accuracy will be restricted by the regularization error. The first results on computational methods for UC typically considered the elliptic Cauchy problem, where the continuation is made from a boundary where both Dirichlet and Neumann data are known into the domain, subject to a second order elliptic operator. Early work focused on rewriting the problem as a boundary integral \cite{CM79, IYH91}, while the earliest finite element reference appears to be \cite{FM86}. The dominating regularization techniques are Tikhonov regularization \cite{TA77} and quasi reversibility \cite{LL69}. Computational methods for the discretization of the regularized problem have been proposed in \cite{RHD99,Bou05, DHH13,BR18, BC20}. 

Recently stabilization techniques introduced for well-posed, but numerically unstable problems \cite{DD76, HFB86, BH04}, were applied to the unique continuation problem in a discretize first, then optimize framework \cite{Bu14, Bu16, BHL18, BO18}.
The stabilization techniques allow for the design of consistent regularization methods, and the upshot is that the 
so-called conditional stability estimates \cite{alessandrini2009} for the UC can be shown to lead to error estimates that reflect the approximation order of the space and the (optimal) stability properties of the ill-posed problem. In particular quantitative a priori error estimates have been derived in a number of situations with careful analysis of the effect of the physical parameters of the problem on the constants of the estimates \cite{BNO19,BNO20,BNO22}. This has lead to a deeper understanding of the computational difficulty of recovering quantities via UC in different parameter regimes. In these works, UC from some internal subset of the bulk was typically considered, avoiding some technical difficulties associated to the Cauchy problem. Nevertheless, the techniques are applicable to either case. This approach is related to previous work on finite element methods for indefinite problems based on least squares minimization in $H^{-1}$ \cite{BLP97,BLP98,BLP01}. More recent results using least squares minimization in dual norm for ill-posed problems can be found in \cite{CIY22, dahmen2022squares}.

Although these methods perform as well as can be expected given the stability of the ill-posed problem, the estimates are not optimal compared to interpolation and conditional sub-Lipschitz stability can in many cases be prohibitively poor for practical purposes. For instance, if $h$ denotes the mesh-parameter, the logarithmic stability of the global UC problem leads to an error bound of the form $|\log(h)|^{-\alpha}$ for some $\alpha \in (0,1)$. On the other hand it has been known from some time in the analysis community that if additional a priori information on the target quantity is added the stability can improve to Lipschitz \cite{AV05,BV06,Sin07,Bou13,LO16}, hinting at the possibility of optimally converging reconstruction methods in this case. Typically the necessary assumption is that the target quantity is of finite dimension $N$. The loss of stability as the dimension becomes large implies that the constant of the stability estimate may grow exponentially in $N$. Even if these results have been known for some time, to the best of our knowledge,  no numerical methods exploiting such additional a priori knowledge exist. The objective of the present work is to present and analyze such a method for the first time. In particular we show that the Lipschitz stability, derived using finite dimensionality, allows us to show that the accuracy of the approximation is optimal compared with interpolation, in spite of the ill-posedness of the problem. 

We will assume that the Dirichlet trace of the function to be approximated by UC is known to reside in a known space of finite dimension. For this situation we derive a Lipschitz stability estimate in norms suitable for numerical analysis (Section \ref{sec:Lip}). We then design a computational method that can leverage the improved stability to enhance the accuracy of the approximation (Section \ref{sec:FEM}). For this method optimal a posteriori and a priori error estimates are proved. As an application we then consider the case where the unknown quantity 
does not lie in the fixed finite dimensional space, but can be approximated to high accuracy in that space, and show that the same bound holds, up to the accuracy of the finite dimensional approximation of the target quantity (Section \ref{sec:Pert}). The case of perturbations in data are also included in this latter case.  We end the paper with some numerical examples validating the theoretical results  (Section \ref{sec:Num}).

\section{Problem setting}
\label{sed:setting}

Let $\Omega \subset \R^2$, be an open, bounded, polygonal domain, and let $\omega \subset \Omega$ be open and nonempty.
We will consider the unique continuation problem of finding a harmonic function $u$ in $\Omega$ given its restriction in $\omega$.
In general, this problem is very unstable, it is only conditionally logarithmically stable \cite{alessandrini2009}. However, if we know, in addition, that 
$u|_{\p \Omega} \in \VN$ where $\VN$ is a subspace of $H^1(\p \Omega)$ that satisfies $\dim(\VN) = N < \infty$,
then the problem becomes Lipschitz stable and thus computationally viable. 
The purpose of this paper is to introduce and analyze a finite element method that realizes this Lipschitz stability numerically.
 
In particular, we will show that the finite element solution $u_h$ converges to the continuum solution with the estimate
    \begin{align}\label{conv_main}
\norm{u - u_h}_{H^1(\Omega)} \le C h^k \norm{D^{k+1} u}_{L^2(\Omega)},
    \end{align}
where $h > 0$ and $k \ge 1$ are the mesh size and polynomial order of the finite element space, see Theorem \ref{th_apriori} for the precise statement. 

\section{Lipschitz stability}\label{sec:Lip}

Following \cite[Def. 1.3.3.2]{grisvard1985} the Lipschitz regularity of $\p \Omega$ allows for defining the Sobolev spaces $H^{s}(\p \Omega)$ for $s \in [-1,1]$.
In view of \cite[Eq. (1.3.3.2)]{grisvard1985} these spaces inherit the properties of the spaces $H^s(\Omega')$ with $\Omega' \subset \R^{n-1}$ a domain.
In particular, the following interpolation inequality holds 
    \begin{align}\label{sobolev_interp}
\norm{u}_{H^{1/2}(\p \Omega)} 
\lesssim 
\epsilon \norm{u}_{H^{1}(\p \Omega)}
+ \epsilon^{-1} \norm{u}_{L^2(\p \Omega)}, \quad u \in H^1(\p \Omega),\ \epsilon > 0,
    \end{align}
see e.g. \cite[Th. 1.4.3.3]{grisvard1985}.
Here $\lesssim$ indicates that we have hidden a constant $C > 0$ that multiplies the right-hand side and that is independent of $u$ and $\epsilon$.

Let $P$ be a projection on $\VN$, and write $Q = 1 - P$ for the complementary one. 
\begin{theorem}
There holds
    \begin{align}\label{continuum_est}
\norm{u}_{H^1(\Omega)} 
\lesssim
\norm{u}_{L^2(\omega)} 
+ \norm{Q u}_{H^{1/2}(\p \Omega)}, 
+ \norm{\Delta u}_{H^{-1}(\Omega)} 
\quad u \in H^1(\Omega).
    \end{align}
\end{theorem}
\begin{proof}
Let us first consider a function $w \in H^1(\Omega)$ satisfying 
    \begin{align}\label{eq_homog}
\Delta w = 0, 
\quad
Q w = 0,
    \end{align}
and write 
    \begin{align*}
\delta = \norm{P w}_{H^{1/2}(\p \Omega)}, \quad \eta = \norm{w}_{L^2(\omega)}, \quad E = \norm{w}_{H^1(\Omega)}.
    \end{align*}
Then on $\p \Omega$ it holds that $w = P w + Q w = P w$ and the standard energy estimate for the Laplacian gives
$E \lesssim \delta$. The continuity of trace for harmonic functions \cite[Th. 1.5.3.4]{grisvard1985} implies, for $\epsilon>0$ 
    \begin{align}\label{eq:weak_trace}
\norm{w}_{H^{-1/2-\epsilon}(\p \Omega)}
\lesssim 
\norm{w}_{L^2(\Omega)}.
    \end{align}
Further, the stability estimate for elliptic unique continuation, see e.g. \cite[Th. 5.3]{alessandrini2009}, reads
    \begin{align*}
\norm{w}_{L^2(\Omega)} 
\le E \mu \left( \frac{\eta}{E} \right),
    \end{align*}
where $\mu$ is a logarithmic modulus of continuity,
\[
\mu(t)\leq \frac{C}{\log(\frac{1}{t})^{\alpha}}, \mbox{ for } t<1 \mbox{ and some } \alpha \in (0,1).
\]

As all norms are equivalent in the finite dimensional range of $P$, combining these three estimates yields
    \begin{align*}
\delta 
\lesssim 
\norm{Pw}_{H^{-1/2-\epsilon}(\p \Omega)} 
\lesssim 
\norm{w}_{L^2(\Omega)}
\le
E \mu \left( \frac{\eta}{E} \right)
\lesssim \delta \mu\left( \frac{\eta}{E} \right).
    \end{align*}
Cancelling out $\delta$ and applying $\mu^{-1}$ gives
    \begin{align*}
1 \lesssim \frac{\eta}{E},
    \end{align*}
which again implies $E \lesssim \eta$.
We have shown \eqref{continuum_est} for functions $w$ satisfying \eqref{eq_homog}.

Let us now consider an arbitrary $u \in H^1(\Omega)$ and let $v \in H^1(\Omega)$ be the solution of 
    \begin{align*}
\Delta v = \Delta u, \quad v|_{\p \Omega} = Q u.
    \end{align*}
Then $w = u - v$ satisfies \eqref{eq_homog} and 
    \begin{align*}
\norm{u}_{H^1(\Omega)} 
&\le
\norm{w}_{H^1(\Omega)} 
+ \norm{v}_{H^1(\Omega)} 
\lesssim 
\norm{w}_{L^2(\omega)} + \norm{v}_{H^1(\Omega)} 
\lesssim
\norm{u}_{L^2(\omega)} + \norm{v}_{H^1(\Omega)}.
    \end{align*}
Now \eqref{continuum_est} follows from the energy estimate for the Laplacian.
\end{proof}
\begin{remark}
The discussion is restricted to two space dimensions since the bound \eqref{eq:weak_trace} on polygonal domains is only shown in this case in \cite{grisvard1985}. For smooth domains the inequality is known to hold also in higher dimensions \cite[Th. B.2.9 and B.2.7]{hormander2007}. In that case the extension to higher dimensions of the below analysis is straightforward.
\end{remark}
\section{Finite element method}\label{sec:FEM}

We will use the stability estimate \eqref{continuum_est} to analyse  a finite element method for approximation of $u$ satisfying
\begin{equation}\label{eq:PDE_UC}
\begin{array}{rcl}
-\Delta u &=& f \mbox{ in } \Omega, \\
u & = & q  \mbox{ in } \omega, \\
u\vert_{\p \Omega} & \in & \VN.
\end{array}
\end{equation}
Here it is assumed that a solution exists, which is clearly not true for all data $(f, q)$.
The unique continuation problem formulated in Section \ref{sed:setting} corresponds to the case $f=0$.

We let $\mcT$ be a decomposition of $\Omega$ into shape regular simplices $K$ that form a simplicial complex. We let $h_K = \mbox{diam}(K)$ be the local mesh parameter and write $h = \max_{K \in \mcT} h_K$ for the global one. Then we assume that a family of such tesselations $\{\mcT\}_h$ is quasi-uniform. On $\mcT$ we define the standard space of continuous finite element functions
    \begin{align}\label{def_Vh}
V_h = \{v \in H^1(\Omega) : v \vert_K \in \P_k \text{ for $K \in \mcT$} \}.
    \end{align}
Here $\P_k$ is the space of polynomial of degree at most $k \ge 1$ on $K$. Further, we write
    \begin{align*}
V_{h0} = V_h \cap H^1_0(\Omega).
    \end{align*}

The discrete inverse inequality \cite[Lemma 1.138]{ern2004} reads
    \begin{align}\label{ineq_inverse}
\norm{h \nabla u}_{L^2(K)}
\lesssim
\norm{u}_{L^2(K)},
\quad u \in \mathbb P_k.
    \end{align}
The implicit constant above does not depend on $h > 0$. The same will be true for all subsequent inequalities. 
It follows from \eqref{ineq_inverse} that for all integers $m \ge l \ge 0$
    \begin{align}\label{ineq_semicl_inv}
\seminorm{u}_{H^m(\mathcal T_h)} 
\lesssim 
\seminorm{u}_{H^l(\mathcal T_h)},
\quad u \in V_h,
    \end{align}
where the broken semiclassical Sobolev seminorms
    \begin{align*}
\seminorm{u}_{H^m(\mathcal T_h)}^2 
= 
\sum_{K \in \mathcal T_h} \norm{(hD)^m u}_{L^2(K)}^2
    \end{align*}
are used. We write, further, 
    \begin{align*}
H^m(\mathcal T_h) = \{u \in L^2(\Omega) : u|_K \in H^m(K)\},
\quad 
\norm{u}_{H^m(\mathcal T_h)}^2 
= 
\sum_{j=0}^m \seminorm{u}_{H^j(\mathcal T_h)}^2.
    \end{align*}

Consider the bilinear form
    \begin{align*}
a(u,z) = (h\nabla u, h\nabla z)_{L^2(\Omega)},\quad  u,z \in H^1(\Omega),
    \end{align*}
and write $A$ for the corresponding quadratic form. This lower and upper case notation convention is used for all bilinear and quadratic forms in the paper. 
The bilinear form $a$ is scaled so that
    \begin{align*}
A(u) \lesssim \seminorm{u}_{H^1(\mathcal T_h)}^2, \quad u \in H^1(\Omega).
    \end{align*}
In particular, in view of \eqref{ineq_semicl_inv},
    \begin{align}
A(u) \lesssim \norm{u}_{L^2(\Omega)}^2, \quad u \in V_h.
    \end{align}
We use this type of ``semiclassical'' scaling systematically throughout the paper. 

Writing $W_h = V_h \times \mathcal V_N \times V_{h0}$,
our finite element method is given by the Euler--Lagrange equations for the Lagrangian $\mathcal L : W_h \to \R$,  
    \begin{align}
\mathcal L(u, y, z) 
&= 
h^2 \frac12  \norm{u - q}_{L^2(\omega)}^2 
+ \frac12 B(u, y)
+ a(u, z) - h^2(f,z)_{L^2(\Omega)}
\\&\qquad
+ \frac12 J(u) + \frac12 \seminorm{h^2 \Delta u + h^2 f}_{H^0(\mathcal T_h)}^2,
    \end{align}
where
    \begin{align}
B(u,y) 
&= 
h \norm{u - y}_{L^2(\p \Omega)}^2 + h \norm{h \nabla_\p(u - y)}_{L^2(\p \Omega)}^2,
\\
J(u) 
&= \sum_{K \in \mathcal T_h} h \|\jump{h \nabla u}\|_{L^2(\partial K \setminus \p \Omega)}^2.
    \end{align}
Here $\nabla_\p$ is the gradient on $\p \Omega$ and $\jump{\cdot}$ is the jump across element faces.  
The first two terms in $\mathcal L$ can be viewed as imposing the conditions 
    \begin{align}
u|_\omega = q, \quad u|_{\p \Omega} \in \mathcal V_N,
    \end{align}
and the next two terms the equation $-\Delta u = f$, cf. \eqref{eq:PDE_UC}. The last two terms are auxiliary regularization (or stabilization) terms. Their choice is motivated by Lemma~\ref{lem_jump} below. 

The trace inequality with scaling \cite[Eq. 10.3.8]{brenner2008}
    \begin{align}\label{ineq_trace}
h^{1/2} \norm{u}_{L^2(\p K)}
\lesssim  
\norm{u}_{L^2(K)}
+ \norm{h \nabla u}_{L^2(K)},
\quad u \in H^1(K),\ K \in \mathcal T_h,
    \end{align}
implies that 
    \begin{align}
B(u,0) \lesssim \norm{u}_{H^2(\mathcal T_h)}^2, 
\quad u \in V_h + H^2(\Omega),
    \end{align} 
and for
    \begin{align}
S(u) = J(u) + \seminorm{h^2 \Delta u}_{H^0(\mathcal T_h)}^2,
    \end{align}
    \begin{align}\label{S_bound}
S(u) 
\lesssim 
\seminorm{u}_{H^2(\mathcal T_h)}^2 + 
\seminorm{u}_{H^1(\mathcal T_h)}^2,
\quad u \in V_h + H^2(\Omega).
    \end{align}

\begin{lemma}\label{lem_jump}
For $u \in V_h + H^2(\Omega)$ and $z \in H_0^1(\Omega)$ there holds
    \begin{align}
a(u,z) \lesssim S^{1/2}(u) \norm{z}_{H^1(\mathcal T_h)}.
    \end{align}
\end{lemma}
\begin{proof}
We integrate by parts
    \begin{align}
a(u,z) = (h \nabla u, h \nabla z)_{L^2(\Omega)}
=
-\sum_{K \in \mathcal T_h} \int_K h^2 \Delta u \, z dx + 
\sum_{F\in\mathcal{F}_h} h\int_F \jump{h\p_\nu u} \, z dx,
    \end{align}
where $\mathcal F_h$ is the set of element faces in the interior of $\Omega$.
The claim follows from \eqref{ineq_trace}.
\end{proof}

More explicitly, the finite element method is: 
find $(u, y, z) \in W_h$ such that
for all $(v, \eta, w) \in W_h$ there holds
    \begin{align}
\label{FEM_noelim}
a(u,w) 
&= 
h^2(f,w)_{L^2(\Omega)}, 
\\
a(v, z) + b(u, y, v, 0)  &\\ + h^2 (u, v)_{L^2(\omega)} + s(u, v) 
&= 
h^2 (q, v)_{L^2(\omega)} + h^2(f, h^2 \Delta v)_{H^0(\mathcal T_h)}
\\ 
b(u, y, 0, \eta) 
&= 
0.
    \end{align}
Here, following our notation convention, $b$ is the bilinear form corresponding to the quadratic form $B$. In particular, the third equation reads
    \begin{align}
0 = b(u, y, 0, \eta) = -h (u - y, \eta)_{L^2(\p \Omega)} - h(h\nabla_\p(u - y), h \nabla_\p \eta)_{L^2(\p \Omega)}.
    \end{align}
As this holds for all $\eta \in \mathcal V_N$ we see that $y = P u|_{\p \Omega}$, and hence this variable can be eliminated.
This leads to the formulation: find $(u, z) \in V_h \times V_{h0}$ such that for all $(v,w) \in V_h \times V_{h0}$
\begin{equation}\label{eq:FEM}
g(u,z,v,w) 
= 
h^2 (q, v)_{L^2(\omega)} 
+ h^2(f, h^2 \Delta v)_{H^0(\mathcal T_h)}
+ h^2(f,w)_{L^2(\Omega)}, 
\end{equation}
where the bilinear form $g$ is defined by
    \begin{align}
g(u,z,v,w)  & 
= 
a(u,w) + a(v,z) 
+ \tilde b(u,v) + h^2 (u, v)_{L^2(\omega)} 
+ s(u,v),
    \end{align}
and $\tilde B(u) = B(Qu, 0)$.
Observe that if $u \in H^2(\Omega)$ satisfies \eqref{eq:PDE_UC} then \eqref{eq:FEM} holds for this $u$ and $z=0$.
In other words, the system \eqref{eq:FEM} is consistent. 

We introduce the norm
    \begin{align}
\tnorm{u, z}^2 
= 
\tilde B(u) + h^2 \norm{u}_{L^2(\omega)}^2 + S(u)
+ \seminorm{z}_{H^1(\mathcal T_h)}^2, \quad u \in V_h + H^2(\Omega),\ z \in H^1(\Omega).
    \end{align}
This is indeed a norm in view of \eqref{continuum_est}
and the Poincar\'e inequality.

\begin{lemma}\label{lem:infsup}
There holds 
\[
\tnorm{u,z} \lesssim \sup_{(v,w) \in V_h \times V_{h0}} \frac{g(u,z,v,w)}{\tnorm{w,z}},
\quad (u, z) \in V_h \times V_{h0}.
\]
\end{lemma}
\begin{proof}
Taking $(v,w) = (u,-z)$ we have
    \begin{align}
g(u,z,u,-z) = \tnorm{u, 0}^2.
    \end{align}
As $V_{h0} \subset V_h$, we can also take $(v,w) = (z,0)$ and get
    \begin{align}
g(u,z,z,0) = A(z) + h^2(u,z)_{L^2(\omega)} + s(u,z).
    \end{align}
Here we used the vanishing $\tilde b(u,z) = 0$ due to $z|_{\p \Omega} = 0$. It follows from \eqref{S_bound} and \eqref{ineq_semicl_inv} that 
    \begin{align}\label{S_bound2}
S(z) \lesssim \seminorm{z}_{H^1(\mathcal T_h)}^2.
    \end{align} 
Moreover, the Poincar\'e inequality implies that 
    \begin{align}\label{Poincare}
h^2\norm{z}_{L^2(\omega)}^2 \le \seminorm{z}_{H^1(\mathcal T_h)}^2. 
    \end{align} 
We see that 
    \begin{align}
A(z) = \seminorm{z}_{H^1(\mathcal T_h)} \lesssim g(u, z, z, 0) + \tnorm{u, 0}.
    \end{align}
Therefore for $\epsilon >0$ small enough
\[
\tnorm{u,z}^2 \leq g(u,z,u + \epsilon z ,-z).
\]
It remains to observe that the inequality
\[
\tnorm{u + \epsilon z,-z} \lesssim \tnorm{u,z}
\]
follows from $\tilde B(z) = 0$ together with \eqref{S_bound2}--\eqref{Poincare}. 
\end{proof}

It follows from Lemma \ref{lem:infsup} that the linear system \eqref{eq:FEM} admits a unique solution for any $f \in L^2(\Omega)$ and $q \in L^2(\omega)$. Writing $(u_h, z_h) \in V_h \times V_{h0}$ for the unique solution, we have the Galerkin orthogonality
    \begin{align}\label{Gortho}
g(u_h - u, z_h, v, w) = 0, \quad (v, w) \in V_h \times V_{h0},
    \end{align}
due to the consistency of \eqref{eq:FEM}.
Here $u$ solves \eqref{eq:PDE_UC}.
Using the orthogonality and Lemma~\ref{lem:infsup} we derive the following best approximation result in the norm $\tnorm{\cdot,\cdot}$.
\begin{lemma}\label{lem:best_approx}
Suppose that $u \in H^2(\Omega)$ and $(u_h, z_h) \in V_h \times V_{h0}$ solve \eqref{eq:PDE_UC} and \eqref{eq:FEM}, respectively. Then there holds
\[
\tnorm{u_h - u,z_h} \lesssim \inf_{\tilde u \in V_h} (\tnorm{u - \tilde u,0} + \|h\nabla (u - \tilde u)\|_{L^2(\Omega)}).
\]
If, furthermore, $u\in H^{k+1}(\Omega)$, with $k \ge 1$ the polynomial order in the definition \eqref{def_Vh} of $V_h$, then
\[
\tnorm{u - u_h,z_h} \lesssim \seminorm{u}_{H^{k+1}(\mathcal T_h)}.
\]
\end{lemma}
\begin{proof}
Let $\tilde u \in V_h$.
In view of the triangle inequality 
\[
\tnorm{u_h - u,z_h} \le 
\tnorm{u_h - \tilde u,z_h}
+ \tnorm{\tilde u - u,0},
\]
it is enough to bound the first term on right-hand side above.
Lemma \ref{lem:infsup} gives
\[
\tnorm{u_h - \tilde u,z_h} \lesssim \sup_{(v,w) \in V_h \times V_{h0}} \frac{g(u_h - \tilde u,z_h,v,w)}{\tnorm{v,w}}.
\]
The Galerkin orthogonality yields
\[
g(u_h - \tilde u,z_h,v,w) = g(u - \tilde u,0,v,w),
\]
and the first claim follows from 
    \begin{align}
g(u - \tilde u,0,v,w)
\lesssim
(\tnorm{u - \tilde u, 0} + \norm{h\nabla(u - \tilde u)}_{L^2(\Omega)}) \tnorm{v,w},
    \end{align}
which again follows from the Cauchy-Schwarz inequality together with 
    \begin{align}
a(u - \tilde u, w) 
\le 
\norm{h\nabla(u - \tilde u)}_{L^2(\Omega)} \tnorm{0, w}.
    \end{align}
The second claim follows from interpolation estimates. 
\end{proof}

The objective is now to use \eqref{continuum_est} to show an a posteriori error estimate for the solutions of \eqref{eq:FEM} and then convergence in the $H^1$-norm for sufficiently smooth solutions. This analysis relies on the following bound
\begin{lemma}\label{lem:Hminus}
Suppose that $u \in H^2(\Omega)$ and $(u_h, z_h) \in V_h \times V_{h0}$ solve \eqref{eq:PDE_UC} and \eqref{eq:FEM}, respectively. 
Then there holds
\[
h^{-1} \norm{h^2 \Delta (u - u_h)}_{H^{-1}(\Omega)} 
\lesssim S^{1/2}(u - u_h).
\]
\end{lemma}
\begin{proof}
By definition
\[
\norm{h^2 \Delta (u - u_h)}_{H^{-1}(\Omega)} = \sup_{\substack{w \in H^1_0(\Omega) \\ \|w\|_{H^1}=1}} a(u-u_h,w).
\]
Taking $v = 0$ in the Galerkin orthogonality \eqref{Gortho} 
we have
    \begin{align}
a(u - u_h, \tilde w) = 0, \quad \tilde w \in V_{h0}.
    \end{align}
Hence for $w \in H^1_0(\Omega)$ and $\tilde w \in V_{h0}$
    \begin{align}
a(u - u_h, w) = a(u - u_h, w - \tilde w).
    \end{align}
Lemma \ref{lem_jump} gives
    \begin{align}
a(u - u_h, w - \tilde w) \lesssim S^{1/2}(u - u_h) \norm{w - \tilde w}_{H^1(\mathcal T_h)}.
    \end{align}
Finally, choosing $\tilde w$ as an interpolant of $w$ we have
    \begin{align}
\norm{w - \tilde w}_{H^1(\mathcal T_h)}
\lesssim \seminorm{w}_{H^1(\mathcal T_h)} \le h \norm{w}_{H^1(\Omega)}.
    \end{align}
\end{proof}

\begin{remark}\label{rem_on_lem_Hminus}
The estimate in Lemma \ref{lem:Hminus} holds under the weaker assumption that $u \in H^2(\Omega)$ satisfies 
\begin{equation}
\begin{array}{rcl}
-\Delta u &=& f \mbox{ in } \Omega, \\
u & = & q  \mbox{ in } \omega.
\end{array}
\end{equation}
That is without the constraint of finite dimensionality on the boundary
Indeed, although the Galerkin orthogonality \eqref{Gortho} does not hold for all $v$ and $w$ under this assumption, it holds in the case that $v = 0$, and only this case is used in the proof. 
\end{remark}

\begin{proposition}\label{prop:apost}(A posteriori error estimate)
Suppose that $u \in H^2(\Omega)$ and $(u_h, z_h) \in V_h \times V_{h0}$ solve \eqref{eq:PDE_UC} and \eqref{eq:FEM}, respectively. 
Then there holds
    \begin{align}\label{apost}
h\|u - u_h\|_{H^1(\Omega)} 
\lesssim 
h \|u_h - q\|_\omega + J^{1/2}(u_h) 
+ \tilde B^{1/2}(u_h) 
+ \seminorm{h^2 \Delta  u_h 
+ h^2 f}_{H^0(\mathcal T_h)}.
    \end{align}
\end{proposition}
\begin{proof}
Applying the estimate \eqref{continuum_est} to the second term of the right-hand side we see that
\[
\norm{u-u_h}_{H^1(\Omega)} 
\lesssim
\norm{u - u_h}_{L^2(\omega)}
+ \norm{Q (u - u_h)}_{H^{1/2}(\p \Omega)} 
+ \norm{\Delta (u - u_h)}_{H^{-1}(\Omega)}.
\]
By definition 
\[
\norm{u - u_h}_{L^2(\omega)} = \norm{q - u_h}_{L^2(\omega)},
\]
and using the interpolation inequality \eqref{sobolev_interp} with $\epsilon = h^{1/2}$ 
\[
h \norm{Q (u - u_h)}_{H^{1/2}(\p \Omega)} = h \norm{Q u_h}_{H^{1/2}(\p \Omega)} \lesssim \tilde B^{1/2}(u_h).
\]
Finally we apply Lemma \ref{lem:Hminus} to deduce that
    \begin{align}
h \norm{\Delta (u - u_h)}_{H^{-1}(\Omega)} 
&
\lesssim
J^{1/2}(u - u_h) 
+ \seminorm{h^2\Delta (u - u_h)}_{H^0(\mathcal T_h)} 
\\&
=
J^{1/2}(u_h) 
+ \seminorm{h^2f + h^2\Delta u_h}_{H^0(\mathcal T_h)}.
    \end{align}
\end{proof}

Observe that the right-hand side of \eqref{apost} can be written 
    \begin{align}
h \|u_h - u\|_\omega + J^{1/2}(u_h - u) 
+ \tilde B^{1/2}(u_h - u) 
+ \seminorm{h^2 \Delta  (u_h - u)}_{H^0(\mathcal T_h)}
    \end{align}
and that this is essentially the same as $\tnorm{u_h - u, 0}$.
Combining Lemma \ref{lem:best_approx} and Proposition \ref{prop:apost} gives 

\begin{theorem}\label{th_apriori} (A priori error estimate)
Suppose that $u \in H^2(\Omega)$ and $(u_h, z_h) \in V_h \times V_{h0}$ solve \eqref{eq:PDE_UC} and \eqref{eq:FEM}, respectively. Then there holds
\[
h \norm{u - u_h}_{H^1(\Omega)} \lesssim \inf_{\tilde u \in V_h} (\tnorm{u - \tilde u,0} + \|h\nabla (u - \tilde u)\|_{L^2(\Omega)}).
\]
If, furthermore, $u\in H^{k+1}(\Omega)$, with $k \ge 1$ the polynomial order in the definition \eqref{def_Vh} of $V_h$, then \eqref{conv_main} holds.
\end{theorem}

\section{Perturbation analysis}\label{sec:Pert}

In many situations additional a priori information is available on the trace of the function. For instance in flow problems the inflow data may be known to be close to a finite set of flow profiles up to smooth perturbations. It is then reasonable to assume that the unknown trace can be approximated to high accuracy as a linear combination of a finite number of known functions. In this section we will show how the above framework can be used to obtain a stable approximation method for this situation. We restrict the discussion to the case that $\Omega$ is convex and $k=1$, corresponding to piecewise affine approximation.

We consider the classical unique continuation problem
\begin{equation}\label{eq:UC_pert}
\begin{array}{rcl}
-\Delta u &=& f \mbox{ in } \Omega, \\
u & = & q   \mbox{ on } \omega.
\end{array}
\end{equation}
Here we assume that $q$ is chosen so that the continuation problem admits a solution. 
We apply the method \eqref{eq:FEM} for the approximation of \eqref{eq:UC_pert}, with the data perturbed by a term $q_\delta$, and call the solution $u_h \in V_h$. The perturbation $q_\delta$ is assumed to satisfy the bound $\|q_\delta\|_{L^2(\omega)} \leq \delta$ for some small $\delta>0$.
In addition, we assume that the trace of the solution $u$ to \eqref{eq:UC_pert} can be approximated by functions in $\VN$. More precisely we assume that there exists $p \in \VN$ such that
    \begin{align}
\|u - p\|_{H^{3/2}(\p \Omega)} \leq \delta.
    \end{align}
We will now sho that if the method \eqref{eq:FEM} is used for the approximation of $u$ then \eqref{conv_main} still holds, up to the perturbation $\delta$. That is, if $(u_h, z_h) \in V_h \times V_{h0}$ is the solution of \eqref{eq:FEM} with $q$ replaced by $q + q_\delta$, then
    \begin{align}\label{conv_pert}
\norm{u - u_h}_{H^1(\Omega)} \lesssim h \norm{D^{2} u}_{L^2(\Omega)} + \delta.
    \end{align}

To show this, we first consider the solution $u^*$ to the problem
    \begin{equation}\label{eq:Pert_prob1}
\begin{array}{rcl}
-\Delta u^* &=& f \mbox{ in } \Omega, \\
u^* & = & p \mbox{ on } \partial \Omega.
\end{array}
\end{equation}
Now let $u^\delta = u - u^*$ and note that $u^\delta$ solves
 \begin{equation}\label{eq:Pert_prob2}
\begin{array}{rcl}
-\Delta u^\delta &=& 0 \mbox{ in } \Omega, \\
u^\delta & = & u - p \mbox{ on } \partial \Omega.
\end{array}
\end{equation}
Elliptic regularity implies
\begin{equation}\label{eq:pert_2}
\|u^\delta\|_{H^2(\Omega)}\lesssim \|u - p\|_{H^{3/2}(\p \Omega)} \lesssim \delta.
\end{equation}
Moreover, since the solution of \eqref{eq:Pert_prob1} has trace in $\VN$ we can apply the method \eqref{eq:FEM} to obtain the approximation $(u_h^*, z_h^*) \in V_h \times V_{h0}$. By Theorem \ref{th_apriori} we know that 
\begin{equation}
\|u^* - u_h^*\|_{H^1(\Omega)} \lesssim h \|D^2 u^*\|_{L^2(\Omega)}.
\end{equation}
Thus
    \begin{align}
\|u^* - u_h^*\|_{H^1(\Omega)} 
\lesssim 
h \|D^2 u\|_{L^2(\Omega)} + 
h \|D^2 u^\delta\|_{L^2(\Omega)}
\le h \|D^2 u\|_{L^2(\Omega)} + h\delta,
    \end{align}
and 
    \begin{align}
\|u - u_h\|_{H^1(\Omega)} 
&\leq 
\|u^\delta\|_{H^1(\Omega)}+\|u^* - u_h^*\|_{H^1(\Omega)}+\|u_h^* - u_h\|_{H^1(\Omega)}
\\&\lesssim
\delta +  h \|D^2 u\|_{L^2(\Omega)} + \|u_h^* - u_h\|_{H^1(\Omega)}.
    \end{align}

We write $e_h = u_h - u_h^*$ and $\eta_h = z_h - z_h^*$, and finish the proof of \eqref{conv_pert} by bounding the $H^1(\Omega)$ norm of $e_h$. We apply Lemma \ref{lem:infsup} to see that
\[
\tnorm{e_h,\eta_h} \lesssim \sup_{(v,w) \in V_h \times V_{h0}} \frac{g(e_h,\eta_h,v,w)}{\tnorm{v,w}}.
\]
Observe that $(e_h, \eta_h)$ satisfies \eqref{eq:FEM} with $q = q_\delta + u^\delta|_\omega$ and $f = 0$.
Therefore using \eqref{eq:FEM} we have that 
$g(e_h,\eta_h,v,w) = h^2(q_\delta + u^\delta,v)_{L^2(\omega)}$
 and hence after application of the Cauchy-Schwarz inequality and \eqref{eq:pert_2} we have
\[
\tnorm{e_h, \eta_h} \lesssim h (\|q_\delta\|_{L^2(\omega)} + \|u^\delta\|_{L^2(\omega)}) \lesssim h \delta.
\]
Applying now \eqref{continuum_est} to $e_h$ we obtain, in view of Remark \ref{rem_on_lem_Hminus},
\[
\|e_h\|_{H^1(\Omega)} \lesssim h^{-1} \tnorm{e_h,0} \lesssim \delta.
\]
We have shown the bound \eqref{conv_pert}.
\begin{remark}
Observe that for the estimate \eqref{conv_main} to be useful, the a priori knowledge can not be based on possible approximation properties of the space $\VN$ alone. Even if $\delta$ can be made smaller by choosing $N$ larger through approximation, the implicit constant in \eqref{conv_pert} grows in $N$. So even if $\delta$ could be made to decay, the discretization error part may blow up. 
\end{remark}

\section{Numerical results}\label{sec:Num}

Our computational implementation of the method is based on the formulation \eqref{FEM_noelim}, where the variable $y \in \mathcal V_N$ has not been eliminated. 
According to our computational examples, the stabilizing bilinear form $s$ in \eqref{FEM_noelim} is actually not needed -- the method appears to converge with the optimal rate even without it. We leave it as a future work to understand if $s$ can indeed be omitted. 

In all the computational examples below, we take $\Omega = [0,1]^2$ and let $\omega$ be the region touching the left, bottom, and right sides of the unit square given by
    \begin{align}
\omega = \{ (x,y) \in \Omega : \text{$x < 1/10$ or $x > 9/10$ or $y < 1/4$}\}.
    \end{align}
We write $\Gamma \subset \p \Omega$ for the top edge of the unit square, and let $\mathcal V_N$ be the subspace of $H^1(\partial \Omega)$ that is spanned by the functions $\phi_n$, $n=1,\dots,N$,
where $\phi_n$ is $\sqrt{2} \sin(n\pi x)$ on $\Gamma$ and vanishes on $\p \Omega \setminus \Gamma$.
Further, the polynomial order $k$ of the finite element spaces $V_k$, see \eqref{def_Vh}, is one. 

To see the effect of the stabilizing term $s$ in the computations, we rescale it by a constant $\gamma \ge 0$. As long as $\gamma > 0$, this makes no theoretical difference.
The $H^1(\Omega)$ error for the manufactured solution 
    \begin{align}
\label{u_simple}
u(x, y) = y \sin{\left(\pi x \right)}
    \end{align}
as a function of the mesh parameter $h$ is shown for different values of $\gamma$ in Figure~\ref{fig_7ada3d8}. We have chosen $N = 5$. Observe that $u|_\Gamma \in \mathcal V_1$. 
We take $\gamma = 0$ in the subsequent examples, since the computational results improve as $\gamma \to 0$.

\begin{figure}
\centering
\includegraphics[width=0.7\textwidth]{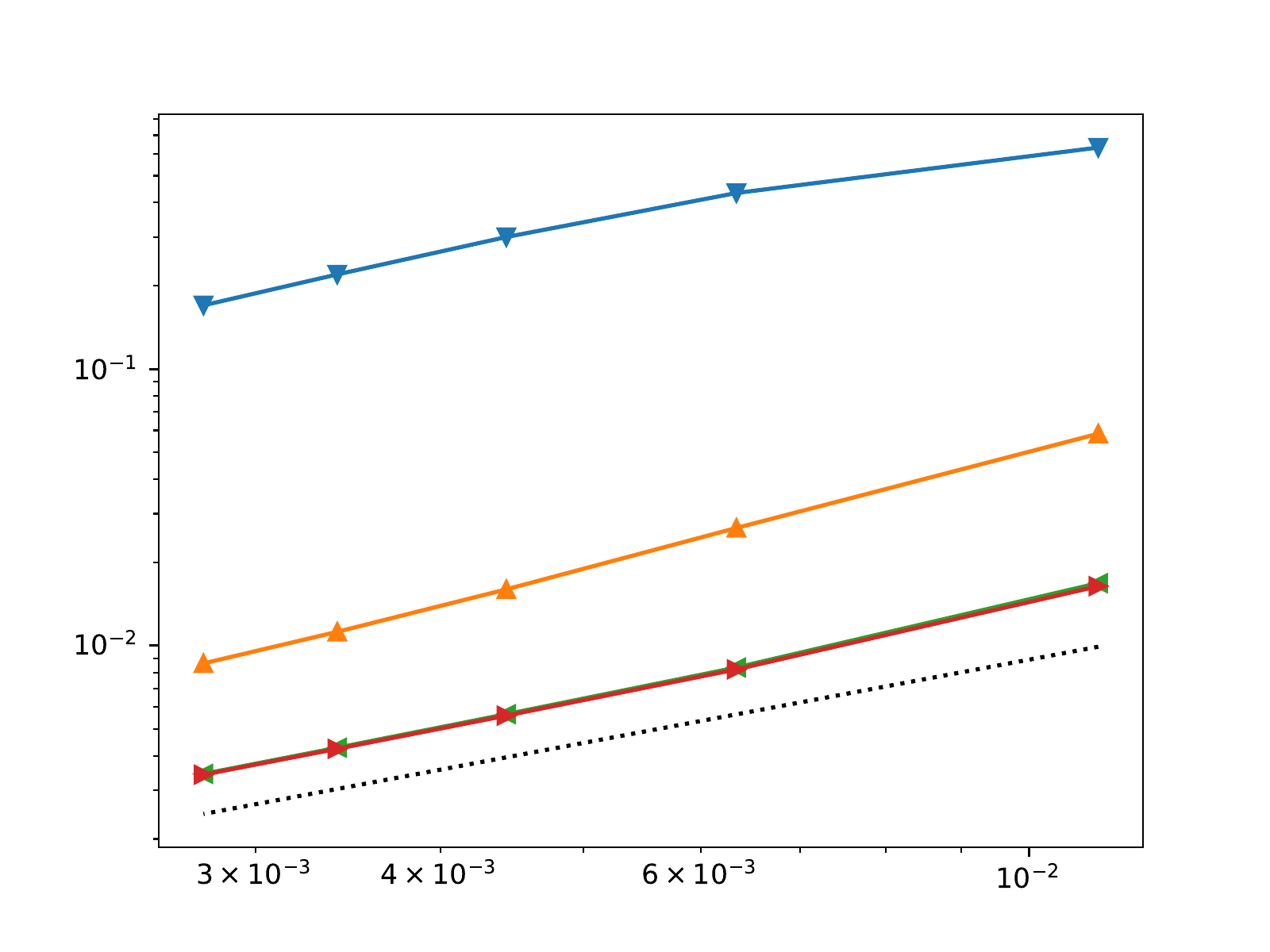}
\caption{The error $\norm{u - u_h}_{H^1(\Omega)}$ as a function of mesh size $h$. Here
$\gamma = 1, 10^{-2}, 10^{-4}, 0$ with triangles pointing down, up, left, and right, respectively. 
Reference rate $h$ in dashed, as predicted by \eqref{conv_main} with $k=1$.}
\label{fig_7ada3d8}
\end{figure}

Let us now illustrate \eqref{conv_pert}, that is, the error estimate with perturbations.
We take 
    \begin{align}
\tilde u(x,y) = y \sin{\left(\pi x \right)} + \frac{1}{100} y \sin{\left(2 \pi x \right)}.
    \end{align}
Then $\tilde u|_{\p \Omega} \in \mathcal V_2$.
Figure \ref{fig_018855c} shows that the convergence stalls for $N = 1$ but not for $N = 2$, in line with the theory.

\begin{figure}
\centering
\includegraphics[width=0.7\textwidth]{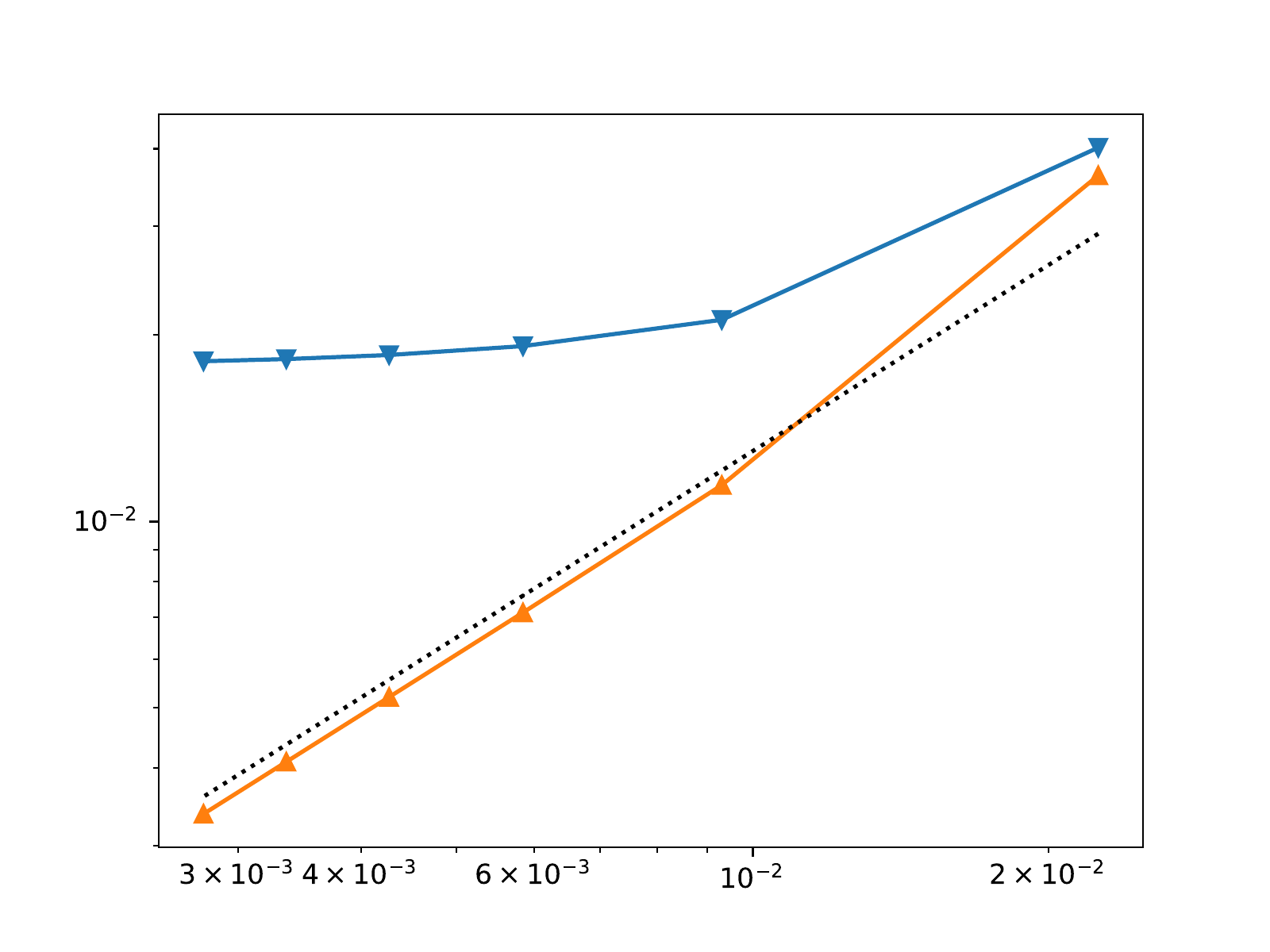}
\caption{
The error $\norm{\tilde u - \tilde u_h}_{H^1(\Omega)}$ as a function of mesh size $h$. 
Here $N = 1, 2$ with triangles pointing down and up, respectively. 
Reference rate $h$ in dashed.}
\label{fig_018855c}
\end{figure}

Our final computational example studies the effect of increasing the number $N$ of degrees of freedom on the boundary.
It appears that merely increasing $N$ does not affect the results if the solution lies in a fixed $\mathcal V_{N_0}$ with $N_0 \le N$. We don't have a theoretical explanation for this, better than expected, phenomenon.
We will compare this case to the case where the exact solution depends on $N$. 

To this end, let us define the ratio
    \begin{align}
C(u) = \frac{\norm{u - u_{h}}_{H^1(\Omega)}}{h \norm{u}_{H^2(\Omega)}},
    \end{align}
and consider the functions
    \begin{align}
\label{u_N}
u_N(x,y) = y \sin{\left(N \pi x \right)}, \quad N=1,2,3,4.
    \end{align}
Then $u_N|_{\p \Omega} \in \mathcal V_N$.
The $H^1(\Omega)$ errors for the functions $u_N$ are plotted in Figure~\ref{fig_bc00789}.
Moreover, the ratio $C(u)$ for $u$ as in \eqref{u_simple},
as well as for $u_N$, is shown in Figure \ref{fig_38d4371} as a function of $N$.
For any $u$, the ratio $C(u)$ gives a lower bound for the constant $C > 0$ in \eqref{conv_main} with $k=1$. We see that the constant grows as a function of $N$, as expected.

\begin{figure}
\centering
\includegraphics[width=0.7\textwidth]{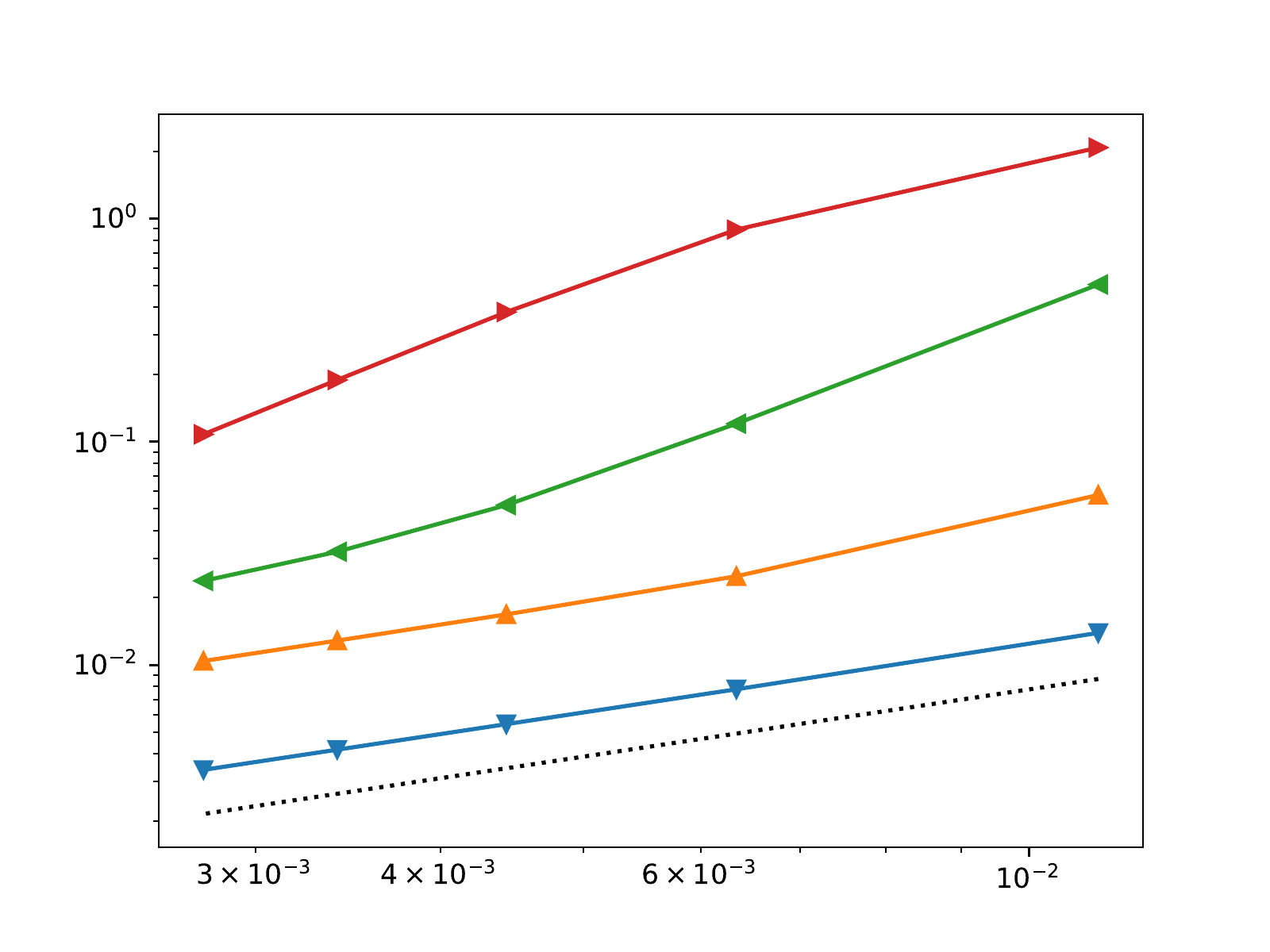}
\caption{The error $\norm{u_N - u_{Nh}}_{H^1(\Omega)}$ error as a function of mesh size $h$.
Here $N = 1, 2, 3, 4$ with triangles pointing down, up, left, and right, respectively. 
Reference rate $h$ in dashed.}
\label{fig_bc00789}
\end{figure}

\begin{figure}
\centering
\includegraphics[width=0.7\textwidth]{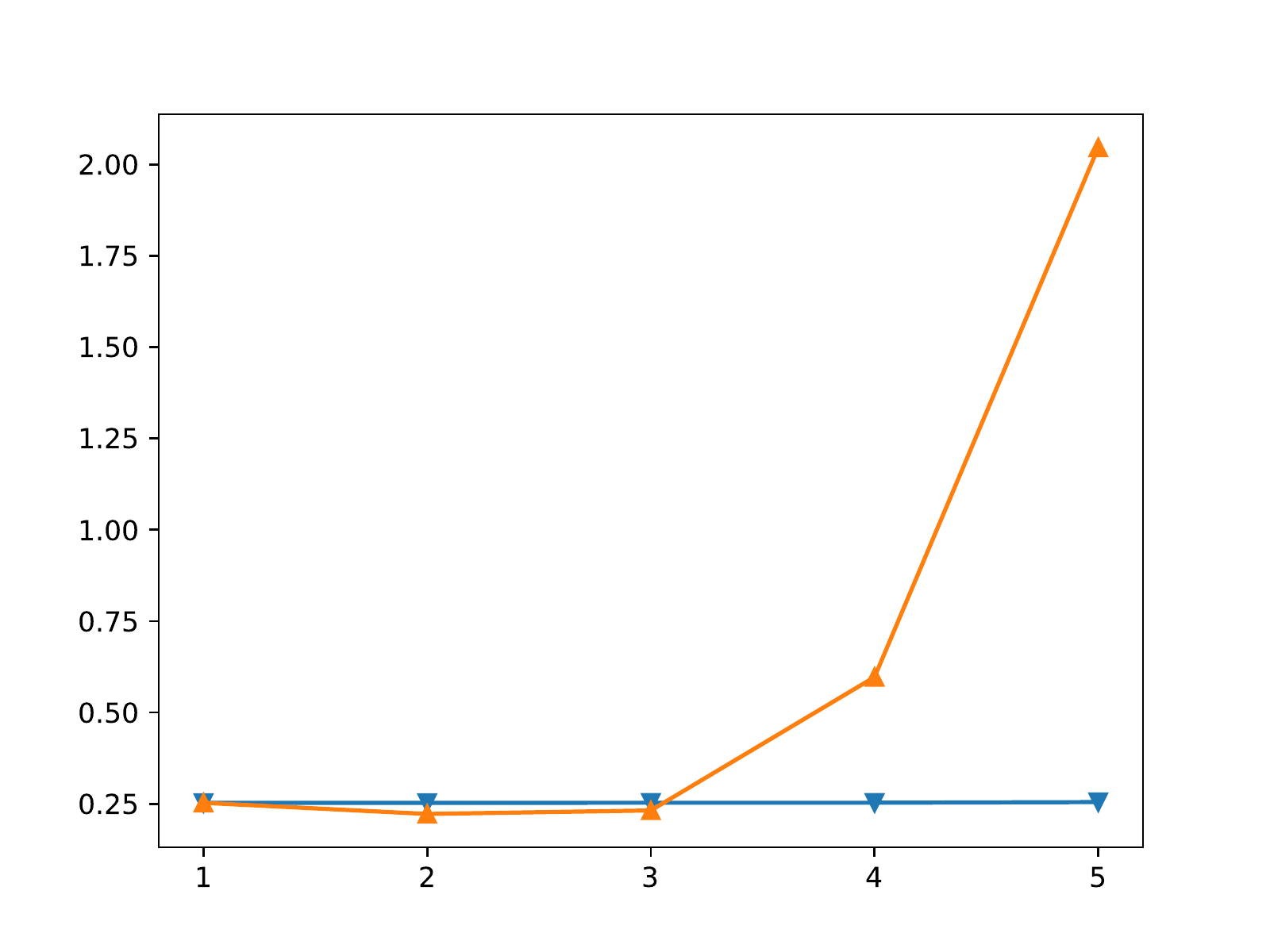}
\caption{The ratio $C(u)$ as a function of $N$. Here $u$ is as in \eqref{u_simple} and \eqref{u_N} with triangles pointing down and up, respectively.}
\label{fig_38d4371}
\end{figure}

\section*{Acknowledgement}
EB acknowledges funding from EPSRC grants EP/T033126/1 and EP/V050400/1. LO acknowledges funding from Academy of Finland grants no. 347715 and 353096

\bibliographystyle{abbrv}
\bibliography{master}
\ifoptionfinal{}{
}
\end{document}